\documentclass[psamsfonts,11pt]{amsart}
\usepackage{setspace}
\usepackage{geometry}
\usepackage{amssymb,amsfonts}
\usepackage[all,arc]{xy}
\usepackage{enumerate}
\usepackage{mathrsfs}
\usepackage{physics}
\usepackage{quiver}
\usepackage{amsmath,amsthm}
\usepackage{comment}
\usepackage{tikz}
\usepackage{amsmath}
\usetikzlibrary{arrows.meta, decorations.pathmorphing}
\usepackage{float}

\newcommand*\deffont[1]{\textit{#1}}

\tikzcdset{
  arrow style=tikz,
  diagrams={>={Computer Modern Rightarrow[scale=0.8]}}
}

\usepackage{hyperref}  
\hypersetup{%
  bookmarksnumbered=true,%
  bookmarks=true,%
  colorlinks=true,%
  linkcolor=blue,%
  citecolor=blue,%
  filecolor=blue,%
  menucolor=blue,%
  pagecolor=blue,%
  urlcolor=blue,%
  pdfnewwindow=true,%
  pdfstartview=FitBH}

\def\makeautorefname#1#2{\expandafter\def\csname#1autorefname\endcsname{#2}}
\def\equationautorefname~#1\null{(#1)\null}
\makeautorefname{footnote}{footnote}%
\makeautorefname{item}{item}%
\makeautorefname{figure}{Figure}%
\makeautorefname{table}{Table}%
\makeautorefname{part}{Part}%
\makeautorefname{appendix}{Appendix}%
\makeautorefname{chapter}{Chapter}%
\makeautorefname{section}{Section}%
\makeautorefname{subsection}{Section}%
\makeautorefname{subsubsection}{Section}%
\makeautorefname{theorem}{Theorem}%
\makeautorefname{thm}{Theorem}%
\makeautorefname{sta}{Statement}%
\makeautorefname{cor}{Corollary}%
\makeautorefname{lem}{Lemma}%
\makeautorefname{prop}{Proposition}%
\makeautorefname{pro}{Property}
\makeautorefname{conj}{Conjecture}%
\makeautorefname{conj}{Convention}%
\makeautorefname{defn}{Definition}%
\makeautorefname{notn}{Notation}
\makeautorefname{notns}{Notations}
\makeautorefname{rem}{Remark}%
\makeautorefname{rems}{Remarks}%
\makeautorefname{quest}{Question}%
\makeautorefname{exmp}{Example}%
\makeautorefname{ax}{Axiom}%
\makeautorefname{claim}{Claim}%
\makeautorefname{ass}{Assumption}%
\makeautorefname{asses}{Assumptions}%
\makeautorefname{con}{Construction}%
\makeautorefname{prob}{Problem}%
\makeautorefname{warn}{Warning}%
\makeautorefname{obs}{Observation}%

\newtheorem{thm}{Theorem}[section]

\newtheorem{prop}{Proposition}[section]
\newtheorem{lem}{Lemma}[section]

\theoremstyle{definition}
\newtheorem{defn}{Definition}[section]

\newtheorem{exmp}{Example}[section]

\newtheorem{rem}{Remark}[section]

    \makeatletter
\let\c@obs=\c@thm
\let\c@cor=\c@thm
\let\c@prop=\c@thm
\let\c@lem=\c@thm
\let\c@prob=\c@thm
\let\c@con=\c@thm
\let\c@conj=\c@thm
\let\c@defn=\c@thm
\let\c@notn=\c@thm
\let\c@notns=\c@thm
\let\c@exmp=\c@thm
\let\c@ax=\c@thm
\let\c@pro=\c@thm
\let\c@ass=\c@thm
\let\c@warn=\c@thm
\let\c@rem=\c@thm
\let\c@conv=\c@thm
\let\c@sch=\c@thm
\makeatother

\def\*#1{\mathbf{#1}}
\newcommand{\R}{\mathbb R}
\newcommand{\Z}{\mathbb Z}
\newcommand{\C}{\mathbb C}

\newcommand\mc{\mathcal}

\newcommand\bb{\mathbb}
\newcommand\td{\tilde}

\newcommand\mf{\mathfrak}

\newcommand\del{\partial}

\newcommand\im{\operatorname{im}}

\newcommand\ot{\otimes}
\newcommand\id{\operatorname{id}}

\usepackage{graphicx}

\newcommand{\ld}{\operatorname{\Lambda}}

\title{On Nonprojectedness of Supermoduli with Neveu-Schwarz and Ramond Punctures}

\author{Tianyi Wang}
\address{Department of Mathematics, University of Pennsylvania, Philadelphia, PA 19104, USA}
\email{tywang9@sas.upenn.edu}

\begin{document}

\begin{abstract}
    We study the supermoduli space $\mf {M}_{g,n,2r}$ of Super Riemann Surfaces (SRS) of genus $g$, with $n$ Neveu-Schwarz punctures and $2r$ Ramond punctures. We improve the result of Donagi, Witten, and Ott \cite{DW13, DW14, DO23} by showing that the supermoduli space $\mf M_{g,n,2r}$ is not projected if $g\geq n+5r+3$.
\end{abstract}

\maketitle

\section{Introduction}
Supermanifolds are the natural backgrounds for supersymmetric field theories, and are extensively studied mathematical objects \cite{manin}. In perturbative superstring theory, particles are replaced by strings; hence, if one is interested in calculating scattering amplitudes to certain perturbative orders, then the natural generalization in superstring theory is to consider amplitude contributions coming from Super Riemann surfaces (SRS) with genus up to that order \cite{W-moresuperstringpert}. An example is illustrated by Figure \ref{feynman diagram becomes riemann surfaces}. Moreover, if one wants to talk about insertions of spacetime bosons and fermions, one needs to consider Neveu–Schwarz (NS) and Ramond (R) punctures, respectively, on an SRS, which label two different kinds of vertex-operator insertions. 

\begin{figure}[H]
    \centering

    \tikzset{every picture/.style={line width=0.75pt}} %set default line width to 0.75pt        
    
    \begin{tikzpicture}[x=0.75pt,y=0.75pt,yscale=-1,xscale=1]
    
    %Straight Lines [id:da41305100640399783] 
    \draw    (181.02,93.25) -- (204.96,129.89) ;
    %Straight Lines [id:da3642427413863373] 
    \draw    (181.51,166.52) -- (204.96,129.89) ;
    %Straight Lines [id:da3828188583770935] 
    \draw    (279.24,129.89) -- (302.21,93.98) ;
    %Straight Lines [id:da19861758402825613] 
    \draw    (279.24,129.89) -- (302.7,165.79) ;
    %Straight Lines [id:da27763868486281473] 
    \draw    (204.96,129.89) -- (229.89,129.89) ;
    %Straight Lines [id:da17559139114833355] 
    \draw    (254.32,129.89) -- (279.24,129.89) ;
    %Shape: Ellipse [id:dp639876458877772] 
    \draw   (229.89,130.84) .. controls (229.89,123.48) and (235.36,117.51) .. (242.1,117.51) .. controls (248.85,117.51) and (254.32,123.48) .. (254.32,130.84) .. controls (254.32,138.2) and (248.85,144.17) .. (242.1,144.17) .. controls (235.36,144.17) and (229.89,138.2) .. (229.89,130.84) -- cycle ;
    %Shape: Ellipse [id:dp44447242577507573] 
    \draw   (362.05,94.91) .. controls (365.21,94.91) and (367.76,101.21) .. (367.75,108.97) .. controls (367.75,116.74) and (365.18,123.04) .. (362.03,123.03) .. controls (358.87,123.03) and (356.31,116.73) .. (356.32,108.97) .. controls (356.33,101.2) and (358.89,94.91) .. (362.05,94.91) -- cycle ;
    %Shape: Ellipse [id:dp673773035409303] 
    \draw   (363.39,142.21) .. controls (366.54,142.21) and (369.1,148.51) .. (369.09,156.27) .. controls (369.08,164.04) and (366.52,170.33) .. (363.36,170.33) .. controls (360.2,170.33) and (357.65,164.03) .. (357.66,156.27) .. controls (357.66,148.5) and (360.23,142.21) .. (363.39,142.21) -- cycle ;
    %Shape: Ellipse [id:dp28328952322176404] 
    \draw   (476.96,94.38) .. controls (480.12,94.38) and (482.67,100.68) .. (482.67,108.44) .. controls (482.66,116.21) and (480.1,122.5) .. (476.94,122.5) .. controls (473.78,122.5) and (471.23,116.2) .. (471.23,108.44) .. controls (471.24,100.67) and (473.81,94.38) .. (476.96,94.38) -- cycle ;
    %Shape: Ellipse [id:dp790728105852189] 
    \draw   (476.96,142.21) .. controls (480.12,142.21) and (482.67,148.51) .. (482.67,156.27) .. controls (482.66,164.04) and (480.1,170.33) .. (476.94,170.33) .. controls (473.78,170.33) and (471.23,164.03) .. (471.23,156.27) .. controls (471.24,148.5) and (473.81,142.21) .. (476.96,142.21) -- cycle ;
    %Curve Lines [id:da19607964238546038] 
    \draw    (362.05,94.91) .. controls (390.34,84.44) and (455.82,88.16) .. (476.96,94.38) ;
    %Curve Lines [id:da5675350902935892] 
    \draw    (363.36,170.33) .. controls (395.02,183.3) and (459.83,179.58) .. (478.27,169.8) ;
    %Curve Lines [id:da5720075683500098] 
    \draw    (362.03,123.03) .. controls (389.68,129.62) and (387,134.4) .. (363.39,142.21) ;
    %Curve Lines [id:da594337987253521] 
    \draw    (476.94,122.5) .. controls (447.13,130.68) and (449.14,137.59) .. (476.96,142.21) ;
    %Shape: Arc [id:dp28587097282321194] 
    \draw  [draw opacity=0] (411.98,131.48) .. controls (414.44,129.84) and (417.47,128.85) .. (420.76,128.83) .. controls (424.21,128.8) and (427.39,129.84) .. (429.93,131.6) -- (420.84,142.41) -- cycle ; \draw   (411.98,131.48) .. controls (414.44,129.84) and (417.47,128.85) .. (420.76,128.83) .. controls (424.21,128.8) and (427.39,129.84) .. (429.93,131.6) ;  
    %Shape: Arc [id:dp8043914187833238] 
    \draw  [draw opacity=0] (434.46,128.6) .. controls (432.03,133.37) and (426.55,136.63) .. (420.24,136.51) .. controls (414.07,136.39) and (408.8,133.06) .. (406.44,128.36) -- (420.46,122.93) -- cycle ; \draw   (434.46,128.6) .. controls (432.03,133.37) and (426.55,136.63) .. (420.24,136.51) .. controls (414.07,136.39) and (408.8,133.06) .. (406.44,128.36) ;  
    \end{tikzpicture}
    
    \caption{With particles replaced by strings, a 1-loop Feynman diagram (left) becomes a Super Riemann surface of genus 1 (right) in superstring perturbation theory.}
    \label{feynman diagram becomes riemann surfaces}
\end{figure}
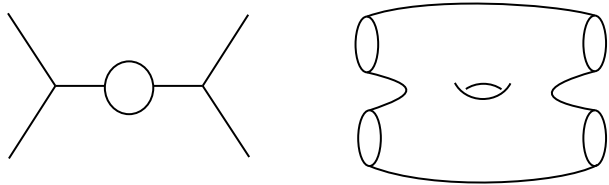

The formulation of superstring theory using the path integral leads naturally to moduli spaces $\mf M_{g,n,2r}$ of SRS with punctures. The role played by Feynman diagrams in ordinary QFT is taken, instead, by (families of) SRS with punctures: one integrates a suitable superstring measure over the supermoduli space of those surfaces. By far, one of the most successful computations was done in 2-loop order \cite{HP1,HP2}, which used the property that $\mf M_g$ is split for $g=2$. However, a major mathematical complication is that for sufficiently large genus $g$, the supermoduli space $\mf M_g$ is not projected, hence does not split. This result was proved by Donagi and Witten in \cite{DW13,DW14}. In fact, they proved an even stronger result when the NS punctures are also present: They showed that $\mf M_{g,n,0}$ is not projected for $g\geq 2$ and $g-1\geq n\geq 1$. Then following this work, Donagi and Ott \cite{DO23} proved that $\mf M_{g,0,2r}$ is not projected for $g\geq 5r+1\geq 6$, i.e., when the R punctures are present, but no NS puncture is present.

The purpose of this paper is to improve the result of Donagi, Witten, and Ott by proving nonprojectedness of supermoduli $\mf M_{g,n,2r}$ for large genus $g$. The main result of this paper is that $\mf M_{g,n,2r}$ is not projected if $g\geq n+5r+3$, which is given in Theorem \ref{main result}.

This paper is structured as follows: In section \ref{section: preliminaries}, we collect several key results in the literature that we will use repeatedly in our proof. In section \ref{section: settings}, we introduce the setup of our proof, and check that several ingredients required in the proof are in place. Our main results, which include a proof of Theorem \ref{main result}, will be established in Section \ref{section: proof of main result}.

\section{Preliminaries}\label{section: preliminaries}
The theory of supermanifolds and SRS has been studied extensively in literature such as \cite{DW13, manin,W-SRSM}. Therefore, in this section, we will only briefly recall relevant definitions and refer the readers to the mentioned literature for details.

\begin{defn}[split supermanifold]
    Let $M$ be a manifold and $V$ a vector bundle over $M$. Then we define the \deffont{split supermanifold} $S:=S(M,V)$ to be the locally ringed space $(M,\mc O_S)$, where $\mc O_S:=\mc O_M\otimes \ld^\bullet V^*$ is the sheaf of $\mc O_M$-valued sections of the exterior algebra $\ld ^\bullet V^*$. In other words,
    \[
    S(M,V)=(M,\mc O_M\ot \ld^\bullet V^*).
    \]
    If we interpret $V$ as a locally free sheaf of $\mc O_M$-modules, then $\mc O_S$ is simply $\ld^\bullet V^*$.    
    Then $\mc O_S$ is $\Z_2$-graded, supercommutative, and its stalks are local rings.
\end{defn}

\begin{exmp}
     $S(\R^m, \mc O_{\R^m}^{\oplus n})$ is just the smooth manifold $\R^m$, with a sheaf of supercommutative ring 
    $$\mc O: U\mapsto C^\infty(U)[\theta^1,...,\theta^n],$$ 
    that sends $U$ to the Grassmann algebra with generators $\theta^1,...,\theta^n$ over $C^\infty(U)=\mc O_{\R^m}(U)$. This is the structure sheaf. An element of the structure sheaf looks like $f=\sum_I f_I\theta^I$ where $I=(i_1,...,i_k)\subset \qty{1,...,n}$ is an index set and $\theta^I=\theta^{i_1}\theta^{i_2}\cdots \theta^{i_k}$. We will call this model space $\R^{m|n}$. Thus $\R^{m|n}$ has even coordinates $x^1,...,x^m$ and odd coordinates $\theta^1,...,\theta^n$. For brevity, we will usually write an element $f$ in the structure sheaf as $f(x^\mu,\theta^\alpha)$, and we will denote the structure sheaf $\mc O$ by $\mc O_{\R^{m|n}}$ if we want to be specific. Similarly, one can define $\C^{m|n}$. 
\end{exmp}

\begin{defn}
    A \deffont{supermanifold} of dimension $m|n$ is a locally ringed space $(S,\mc O_S)$ that is locally isomorphic to some model space $S(M,V)$, where $\dim M=m$ and $\operatorname{rank}V=n$. We will usually just refer to this supermanifold as $S$ with the understanding that it also comes with the structure sheaf $\mc O_S$. The manifold $M$ is called the \deffont{reduced space} or the \deffont{bosonic reduction} of $S$, and we usually write $|S|=M.$ A \deffont{morphism} between supermanifolds is simply a morphism of locally ringed spaces.
\end{defn}

\begin{defn}
    A \deffont{sheaf} $\mc F$ on a supermanifold $S$ is simply a sheaf on the reduced space $M$, and sheaf cohomology $H^k(S,\mc F)$ of $\mc F$ on $S$ is simply the sheaf cohomology $H^k(M,\mc F)$ on the reduced space.
\end{defn}

\begin{defn}[tangent bundle]
    The \deffont{tangent bundle} $T_S$ of supermanifold $S$ is the sheaf of $\mc O_S$-modules given by derivations of the structure sheaf, i.e., as sheaves $$T_S:=\operatorname{Der}(\mc O_S).$$ 
    It is a $\Z_2$-graded vector bundle, or a locally free sheaf of $\mc O_S$-modules.
\end{defn}

In the simplest example where $S=\R^{m|n}$ or $\C^{m|n}$, $T_S$ is the free $\mc O_S$-module generated by even tangent vectors $\pdv{x^\mu}$ for $\mu=1,...,m$ and odd tangent vectors $\pdv{\theta^\alpha}$ for $\alpha=1,...,n$.

In general, for a split supermanifold $S=S(M,V)$, the tangent bundle $T_S$ need not have distinguished complementary even and odd subbundles: The even and odd parts are not sheaves of $\mc O_S$ modules. For example, $\pdv{\theta^1}$ is an odd vector field, but multiplication by an element of $\mc O_S$ might change parity: Consider $\theta^1\pdv{\theta^1}$. This is now an even vector field. However, we note that since elements in $\mc O_M$ are even, if we restrict ourselves to multiplying only elements in $\mc O_M$ then parity is preserved. More specifically, what we are saying is that the restriction $T_S|_M$ of $T_S$ to the reduced space $M$ does split (By restriction to $M$ we mean pullback, under the natural inclusion $i:M\to S$, of a sheaf of $\mc O_S$-modules to a sheaf of $\mc O_M$-modules. In this case, this is accomplished by setting all odd functions to 0). Explicitly, the splitting is given by
\begin{align*}
    T_{S,+}:&=T_M,\\
    T_{S,-}:&=V.
\end{align*}
For a general supermanifold $S$, we only have local isomorphisms between $S$ and $S(M,V)$ for some model space $M$ and vector bundle $V$, but there is no canonical identification $T_{S,-}= V$. However, there is still a way to define $T_{S,-}$. We consider the natural inclusion $i:M\to S$, where we view $M$ as the subspace that is locally given by $\theta^\alpha=0$ for all $\alpha$. Then then we have a normal bundle sequence
\[
0\to T_M\to T_S|_M\to N_{M/S}\to 0.
\]
Now $N_{M/S}$ is a purely odd bundle. Therefore, we may define $T_{S,+}=T_M$ and $T_{S,-}=N_{M/S}$. We note that the rank of $T_{S,+}=T_S$ is $m|0$ and the rank of $T_{S,-}$ is $0|n$ if $\dim S=m|n$.

Next, we come to the definition of split and projected supermanifolds.

\begin{defn}[split and projected supermanifold]
    A supermanifold $S$ is said to be \deffont{split} if it is globally isomorphic to $S(M,V)$ for some manifold $M$ and some vector bundle $V\to M$. In other words, it is globally isomorphic to some split supermanifold.

    A supermanifold $S$ is said to be \deffont{projected} if there exists a projection map $\pi:S\to M$, such that the inclusion $i:M\to S$, which is identity on the reduced space $i: |M|\to |S|$, and $i^*:\mc O_S\to \mc O_M$ given by imposing $\theta^1=\cdots=\theta^n=0$, is a section of $\pi$, i.e., $\pi i=\id_M$ and $i^*\pi^*=\id_{\mc O_M}$.
\end{defn}

\begin{lem}
    Any split supermanifold $S$ is projected.
\end{lem}

\begin{proof}
    Indeed, if $S$ is split, then $\mc O_S\cong \mc O_M\otimes \ld^\bullet V^*$. Then we have a morphism $\pi^*:\mc O_M\to \mc O_S$ via embedding $\mc O_M$ to be the degree zero part. And it is easy to see that $i^*\pi^*=\id_{\mc O_M}$. Now to construct $\pi:S\to M$ with the desired property, we choose the map of underlying space $\pi:|S|\to |M|$ to be identity, and $\pi^*: \mc O_M\to \pi_*\mc O_S$ to be the map given above.
\end{proof}

A natural question to ask is whether we can characterize the obstruction to splitting or projection of a supermanifold via some cohomology class. The answer is affirmative. The obstruction theory of supermanifolds is developed in \cite{manin, gr, ber}, and is summarized in \cite{DW13} with great detail. The result relevant to us is that a necessary condition for a supermanifold $S=(M,\mc O_S)$ to split and to be projected is that a certain cohomology class called the \deffont{second obstruction class}
\[
\omega_2\in H^1(M,T_+\otimes \ld^2 V^*)
\]
vanishes. The same construction is also discussed in \cite{manin}, with a more geometric perspective. The second obstruction class will be useful in concluding that certain supermanifolds are not projected. Indeed, many basic examples of non-projected supermanifolds in \cite{DW13} were established by showing $\omega_2\neq 0$ for those supermanifolds.

However, for complicated supermanifolds, like supermoduli spaces, it is very difficult to directly evaluate $\omega_2$ and show that it is nonzero. Fortunately, there are also some indirect results we can use. The next theorem, which is Corollary 2.8 of \cite{DW13}, says that a finite covering of a non-projected supermanifold is also non-projected.

\begin{thm}\label{cover of nonproject is nonproject}
    Let $\pi:Y\to X$ be a finite covering map of supermanifolds. If $\omega_2(X)\neq 0$, then $\omega_2(Y)\neq 0$, so $Y$ is not projected.
\end{thm}

If we have a submanifold of a supermanifold, which we know is not projected, then it is natural to suspect that the big supermanifold itself is not projected. This is not true in general: For example, a supermanifold of dimension $n|2$ is split if and only if it is projected by Corollary 2.3 of \cite{DW13}. In particular, $\bb{CP}^{n|2}$ is split, but a generic hypersurface of $\bb{CP}^{n|2}$ is nonsplit\footnote{The author thanks Edward Witten for pointing this out.}. However, in certain special cases the statement is true, which is made precise by a modified version of Corollary 2.12 of \cite{DW13}.

\begin{thm}\label{immersion+NBS splits gives non-project}
    Let $a:S'\to S$ be an immersion of supermanifolds, with reduced spaces $M'\subset M$, such that the normal sequence of $M'$ decomposes: $a^*T_M\cong T_M'\oplus N$ , where $N$ is the even component of the normal bundle to the map $a$. If $\omega_2(S')\neq 0$, then we also have $\omega_2(S)\neq 0$, so $S$ is not projected.
\end{thm}

The remaining part of this section will be devoted to the central objects of study in this paper, Super Riemann Surfaces and their moduli. Again, much of the theory below is established in \cite{W-SRSM} and \cite{DW13}. We will only collect relevant results here and refer the readers to the literature mentioned for detailed explanations.

\begin{defn}
    Let $S$ be a supermanifold. A distribution (i.e., a subsheaf of the tangent sheaf) $\mc D\subset T_S$ is called a \deffont{superconformal structure} if it is
    \begin{itemize}
        \item an odd distribution, i.e., a subbundle of rank $0|1$.
        \item \deffont{everywhere non-integrable}. By Frobenius' theorem, a distribution is integrable if it is closed under the Lie bracket. Since $\mc D$ is an odd distribution, its Lie bracket will be denoted by the anticommutator. By assumption $\mc D$ is  of rank 0|1, so it will be generated by a single odd vector field $v$. Since $v^2:=\frac{1}{2}\{v,v\}$, we define $\mc D$ to be \deffont{integrable} if $v^2\in \mc D$, and define it to be \deffont{everywhere non-integrable} if $v^2$ is everywhere independent of $v$ over $\mc O_S$.
    \end{itemize}
\end{defn}

\begin{defn}
    A \deffont{Super Riemann Surface} (SRS) is a pair $(S,\mc D)$, where $S=(C,\mc O_S)$
    is a complex compact supermanifold of dimension $1|1$, and $\mc D$ is a superconformal structure.
\end{defn}

\begin{rem}
    Since $S$ is of dimension $1|1$, $T_S$ has rank $1|1$. Since $\mc D$ has rank $0|1$ and $\mc D^2$ is even, hence has rank $1|0$, together they generate the entire tangent bundle $T_S$. Moreover, we have an isomorphism $T_S/\mc D\cong \mc D^2$. Therefore, we actually have an exact sequence of sheaves
    \[
    0\to \mc D\to T_S\to \mc D^2\to 0.
    \]
\end{rem}

The next result is from \cite{DW13}, which gives direct characterizations of the generating vector field of the superconformal structure.
\begin{lem}\label{v in superconformal coord}\label{v^2=d/dx}
    Locally on a SRS one can choose coordinates $z$ and $\theta$, called a \deffont{superconformal coordinate}, such that $\mc D$ is generated by the vector field
    \[
    v:=\pdv{\theta}+\theta \pdv{z}.
    \]
\end{lem}

Now we can do deformation theory on SRS. We must therefore consider automorphisms on the SRS. Locally, the infinitesimal automorphisms are generated by \deffont{superconformal vector fields}, which preserve the distribution $\mc D$. In superconformal coordinates, a short calculation shows that the even superconformal vector field takes the form
\begin{align}\label{even supconf vec}
    \chi^+=f(z)\pdv{z}+\frac{f'(z)}{2}\theta \pdv{\theta}
\end{align}
while the odd one takes the form
\begin{align}\label{odd superconf vec}
    \chi^-=-g(z)\qty(\pdv{\theta}-\theta \pdv{z}),
\end{align}
where $f,g$ are holomorphic (even) functions on $S$ that depends on $z$ only, not $\theta$.
We denote the sheaf of superconformal vector fields on $S$ by $\mc A_S$, which is also the sheaf of infinitesimal automorphisms of $S$.
Now we may consider the moduli space $\mf M_g$ of SRS of genus $g$. Given a point $S\in \mf M_g$, we want to know what the tangent space $T_S\mf M_g$ is. A standard argument in deformation theory shows
\[
T_S\mf M_g=H^1(S,\mc A_S).
\]
The notion of a \deffont{puncture} on an ordinary Riemann surface has two analogs on an SRS. In string theory, they are known as the \deffont{Neveu-Schwarz puncture} (NS) and \deffont{Ramond puncture} (R). An NS puncture on an SRS $S$ is the obvious analog of a puncture on an ordinary Riemann surface: it is simply the choice of a point in $S$. 

If $S$ is a SRS with $n$ marked points $p_1,...,p_n$, or in divisor form $P=p_1+\cdots+p_n$, then it is an element of $\mf M_{g,n,0}$. The infinitesimal automorphisms on $S$ must preserve the superconformal structure $\mc D$, and preserve the marked points, which imposes an extra condition on superconformal vector fields of the form \eqref{even supconf vec} and \eqref{odd superconf vec} that they must vanish on the marked points, i.e., $f(0)=g(0)=0$ in local coordinate if the marked point is given by the equation $z=0$. Therefore, we conclude that
\[
T_S\mf M_{g,n,0}=H^1(S,\mc A_S(-P)).
\]
Now we introduce the second kind of puncture: \deffont{Ramond puncture}. In this situation, we assume the underlying supermanifold $(C,\mc O_S)$ is still smooth, but the odd distribution $\mc D$ is no longer everywhere non-integrable. The generator $v$ in the local form in Lemma \ref{v in superconformal coord} is replaced by
\[
v:=\pdv{\theta}+z\theta \pdv{z}.
\]
In other words, $\mc D$ fails to be a maximally nonintegrable distribution along the divisor $z=0$. Therefore, a SRS with a Ramond puncture is technically no longer a SRS. We can also have multiple Ramond punctures, by which we mean $\mc D$ fails to be a distribution along multiple divisors, and near each divisor we can find local coordinates as described above. The topology of SRS restricts the number of Ramond punctures to always be even.
In \cite{DW13}, it was shown that in the presence of Ramond punctures $R=q_1+\cdots+q_{2r}$, the sheaf of superconformal vector fields is given by
\begin{equation}\label{autosheaf with R punctures}
    \mc A_S\cong (T_S/\mc D)\otimes \mc O_S(-R).
\end{equation}
Therefore, for a SRS $S$ with NS punctures $P=p_1+\cdots+p_n$ and R punctures $R=q_1+\cdots+q_{2r}$, we still have
\[
T_S\mf M_{g,n,2r}=H^1(S,\mc A_S(-P)),
\]
but now $\mc A_S$ is given by \eqref{autosheaf with R punctures}.

\section{Settings}\label{section: settings}
Our setup will be the following: Let $\pi:Y\to X$ be a branched cover of SRS. We use $g$ to denote the genus of $Y$. Let us fix $g_0=2$ to be the genus of $X$ for the rest of the paper, and we require $X$ to have only 1 branch point. Let $d$ be the degree of $\pi$. If $p\in X$ the the branch point, and $\pi^{-1}(p)=\{q_1,...,q_s\}$. Let $a_j$ denote the local degree of $\pi$ at $q_j$, for $1\leq j\leq s$. Then we define the \deffont{ramification pattern} $\rho=(a_1,...,a_s)$. Ramification points with odd local degree will correspond to NS punctures on $Y$, while ramification points of even local degree will correspond to R punctures on $Y$ after blow ups. See section 3.4 of \cite{DW13} for details. There will always be an even number of R punctures \cite{DO23}, so we can denote the number of R punctures on $Y$ by $2r$, and let $n$ denote the number of NS punctures on $Y$. So $n+2r=s$.

Now we allow the curves $Y,X$ to vary continuously, hence the covering map $\pi:Y\to X$ and the branch point in $X$ also vary continuously. But we require the genera $g,g_0$ of $Y,X$, and the ramification patter $\rho$ to be fixed throughout the process. There is a moduli space $\mf M_{d,\rho}$ parameterizing all such branched coverings. 

The map
\[
\Phi: \mf M_{d,\rho}\to \mf M_{2,1,0}\qquad (\pi:Y\to X)\mapsto X
\]
is a finite covering by Lemma 14 of \cite{DO23}, and it was already established that $\mf M_{2,1,0}$ is not projected by \cite{DW13}. Therefore, $\mf M_{d,\rho}$ is not projected by Theorem \ref{cover of nonproject is nonproject}. Moreover, 

\begin{prop}
the map
\[
\Psi: \mf M_{d,\rho}\to \mf{M}_{g,n,2r}\qquad (\pi:Y\to X)\mapsto Y
\]
is an immersion of supermanifolds.
\end{prop}

\begin{proof}
    By Lemma 15 of \cite{DO23}, the composition $F\circ \Psi$ of $\Psi:\mf M_{d,\rho}\to \mf M_{g,n,2r}$ with the forgetful map $F:\mf M_{g,n,2r}\to \mf M_{g,0,2r}$ is an immersion. Hence $\Psi$ itself must be an immersion.
\end{proof}

Moreover, the associated normal bundle sequence of $\Psi$ splits, which we will check. The proof is basically the same as the proof of Proposition 5.2 of \cite{DW13}, but we can modify the proof and make it slightly more explicit by constructing a concrete lifting map of tangent vector fields of the base to the branched cover, in the case where the cover is Galois. Hence, we include the proof here.

\begin{prop}\label{NBS splits}
     Let $\psi:\mc {SM}_{d,\rho}\to \mc{SM}_{g,n,2r}$ denote the bosonic reduction of the map $\Psi$. Then the induced normal bundle sequence on the reduced space
    \[
    0\to T_{\mc {SM}_{d,\rho}}\to \psi^*T_{\mc {SM}_{g,n,2r}}\to  N\to 0
    \]
    is a split exact sequence.
\end{prop}
\begin{proof}
    We pick a branched cover $\pi:(Y,\mc L_Y)\to (X,\mc L_X)$ of spin curves, representing a point in $\mc{SM}_{d,\rho}$. First, we assume that the cover is $G$-Galois. Under $\psi$ this point goes to $(Y,\mc L_Y)$. We note that if $R_Y$ is the divisor on $Y$ corresponding to the Ramond punctures, then $\mc L_Y^2\cong K_Y(-R_Y)$. Deformation theory gives an identification of $\psi$ at this point:
    \[
    \psi:H^1(X,T_X(-P_X)) \to H^1(Y,T_Y(-P_Y-R_Y)).
    \]
    Since $\psi$ takes the deformation of the base to the corresponding uniquely determined deformation of the branched cover, it is induced by lifting vector fields. In fact, there is an injection of sheaves
    \begin{equation}\label{bosonic L lift}
        L:T_X(-P_X)\to \pi_*T_Y(-P_Y-R_Y)
    \end{equation}
    whose induced map on $H^1$ is $\psi$. To see that $L$ is an injection, we cover $X$ locally by small open sets $X=\bigcup_\alpha U_\alpha$, such that $U_\alpha$ and $\pi^{-1}(U_\alpha)$ only contain at most one branch (or ramification) point for all $\alpha$. If $U_\alpha$ does not contain any marked points, then by choosing sufficiently small open covers, we may assume $\pi:\pi^{-1}(U_\alpha)\to U_\alpha$ is an isomorphism, hence there is no problem constructing $L$ on $U_\alpha$. 

    Now we analyze the situation where $q\in \pi^{-1}(U_\alpha)$ is a ramification point and $p=\pi(q)\in U_\alpha$ is a branch point. Let $e_q=k>1$ be the local degree of $q$. Then locally we may choose holomorphic coordinates $w$ on $Y$ and $z$ on $X$ such that $w(q)=z(p)=0$, and such that locally $\pi$ is given by $z=w^k$. To construct $L$, locally a section in $\Gamma(U_\alpha,T_X(-P_X))$ is of the form
    \[
    \chi=f(z)\pdv{z}
    \]
    where $f$ is a holomorphic function with $f(0)=0$. A section in $\Gamma(U_\alpha,\pi_*T_Y(-P_Y-R_Y))$ is of the form
    \[
    \td\chi=g(w)\pdv{w}
    \]
    with $g(0)=0$, where now $\td\chi$ is viewed as a vector field on $\pi^{-1}(U_\alpha)\subset Y$. Now the condition that $\td\chi$ is a lift of $\chi$, namely $\pi_*\td\chi=\chi$, reads
    \[
    \pi_*\td\chi=g(w)\pdv{z}{w}\pdv{z}=kw^{k-1}g(w)\pdv{z}=f(w^k)\pdv{z}=\chi,
    \]
    which implies
    \[
    g(w)=f(w^k)/kw^{k-1}.
    \]
    Note that the above expression is a well-defined holomorphic function: since $f(0)=0$, we have $f(w^k)=w^kh(w)$ for some holomorphic $h$, and furthermore $g(0)=0$. It is also clear from the expression that $g$ is uniquely determined by $f$, hence $\td\chi$ is uniquely determined by $\chi$. Thus, we can also construct $L$ on this small neighborhood $U_\alpha$, which is injective. Therefore, this construction gives rise to an injection by lifting infinitesimal automorphisms of vector fields on sufficiently small open sets near each ramification point, and these lifts are compatible on the intersections of small open sets. Hence, this glues to an injection of sheaves \eqref{bosonic L lift}. The induced map on $H^1$ is precisely
    \[
    \psi:H^1(X,T_X(-P_X))\to H^1(Y,T_Y(-P_Y-R_Y)).
    \]
    We also note that if $\pi:Y\to X$ is $G$-Galois and $\td\chi$ is a lift of $\chi\in \Gamma(X,T_X)$, then $\td\chi$ must be a $G$-invariant vector field. Hence, the lift in \eqref{bosonic L lift} splits
    \[
    L:T_X(-P_X)\to \pi_*T_Y(-P_Y-R_Y)^G\oplus\mc Q
    \]
    into the $G$-invariant part and some other part $\mc Q$. The inclusion $$T_X(-P_X)\to \pi_*T_Y(-P_Y-R_Y)^G$$ is actually an isomorphism, where the isomorphisms are given by lift and projection. Hence taking $H^1$ we have
    \[
    \psi:H^1(X,T_X(-P_X))\to H^1(Y,T_Y(-P_Y-R_Y))^G\oplus H^1(X,\mc Q)
    \]
    where we used 
    \[
    H^1(X,\pi_*T_Y(-P_Y-R_Y)^G)=H^1(Y,T_Y(-P_Y-R_Y)^G)\cong H^1(Y,T_Y(-P_Y-R_Y))^G
    \]
    and $\psi$ is given by inclusion into the $G$-invariant part in the summand. Hence, it immediately follows that the normal bundle sequence splits.

    Finally, in the case where the covering is not Galois, we pass to the Galois closure of $\pi:Y\to X$. Let $\hat{Y}$ be the $G$-Galois closure of $Y$. Then there is a covering $\hat{\pi}:\hat{Y}\to Y$ with Galois group $H$, where $H<G$ is the stabilizer subgroup of an unramified point of $Y$. The pullback $\hat{\pi}^*$ includes cohomologies of $Y$ as the $H$-invariant part of cohomologies of $\hat{Y}$:
    \[
    \hat{\pi}^*: H^1(Y,T_Y(-P_Y-R_Y))\cong H^1(\hat{Y},T_{\hat{Y}}(-P_{\hat{Y}}-R_{\hat{Y}}))^H\hookrightarrow H^1(\hat{Y},T_{\hat{Y}}(-P_{\hat{Y}}-R_{\hat{Y}})).
    \]
    Introducing the notation $D_{\hat{Y}}=P_{\hat{Y}}+R_{\hat{Y}}$ and define $D_Y,D_X$ similarly, we have a commutative diagram with exact rows given by the normal bundle sequences evaluated at corresponding fibers:
    \[\begin{tikzcd}
	0 & {H^1(X,T_X(-D_X))} & {H^1(Y,T_Y(-D_Y))} & N & 0 \\
	0 & {H^1(X,T_X(-D_X))} & {H^1(\hat{Y},T_{\hat{Y}}(-D_{\hat{Y}}))} & {\hat{N}} & 0
	\arrow[from=1-1, to=1-2]
	\arrow[from=1-2, to=1-3]
	\arrow[no head, from=1-2, to=2-2]
	\arrow[shift left, no head, from=1-2, to=2-2]
	\arrow[from=1-3, to=1-4]
	\arrow["{\hat{\pi}^*}"', hook, from=1-3, to=2-3]
	\arrow[from=1-4, to=1-5]
	\arrow["i", hook, from=1-4, to=2-4]
	\arrow[from=2-1, to=2-2]
	\arrow[from=2-2, to=2-3]
	\arrow[from=2-3, to=2-4]
	\arrow[from=2-4, to=2-5]
    \end{tikzcd}\]
    where $i:N\to \hat{N}$ is the unique map that makes the square commute. A simple diagram chase shows that $i$ is injective, hence the spaces in the upper row can be viewed as subsets included in the corresponding spaces in the bottom row. By the previous argument, we already know that there exists a splitting $s:\hat{N}\to H^1(\hat{Y},T_{\hat{Y}}(-D_{\hat{Y}}))$. Restrict this splitting to the subspace gives an induced splitting $N\to H^1(Y,T_Y(-D_Y))$, concluding the proof.
\end{proof}

Hence, we conclude that $\mf M_{g,n,2r}$ is not projected by Theorem \ref{immersion+NBS splits gives non-project}. The necessary and sufficient condition for such an immersion $\Psi$ to exist, or equivalently for the tuple $(g,n,r)$ to be \deffont{realizable}, using the terminology of \cite{DO23}, is that the genus $g$ determined by the Hurwitz formula
\begin{equation}\label{hurwitz formula}
    g=1+d(g_0-1)+\frac{1}{2}\sum_{j=1}^{s}(a_j-1)
\end{equation}
is nonnegative, where we recall that $s=n+2r$, and $\rho=(a_1,...,a_{s})$ is the ramification pattern of $\pi$, with each $a_j$ a local degree such that $\sum_j a_j=d$. Moreover, Theorem 4 of \cite{huse} ensures this is the only constraint: As long as the configurations $\rho,d,g_0$ make $g\geq 0$ in \eqref{hurwitz formula}, there exists a branched cover $\pi:Y\to X$ with the specified behavior.

Our next task is to determine, to the best of our ability, the condition for the tuple $(g,n,r)$ to be realizable. This is given by Theorem \ref{main theorem} and Theorem \ref{main result} in the next section.

\section{Proof of Main Result}\label{section: proof of main result}
Using the minimal model above with $g_0=2$ and only one branched point on $X$, we can prove our first nonprojectedness theorem for supermoduli space. The proof is combinatorial.

\begin{thm}\label{main theorem}
    Let $g,n,r$ be positive integers. The supermoduli space $\mf M_{g,n,2r}$ is not projected if the following two conditions are met:
    \begin{enumerate}
        \item [(1)] genus bound: $g\geq n+5r+1$;
        \item [(2)] congruence condition: $2g-2+n+2r\equiv 0\mod 3$.
    \end{enumerate}
\end{thm}

\begin{proof}
    Substituting $g_0=2$ in \eqref{hurwitz formula} shows that $2g=2+3d-n-2r$. Hence, we must have $2g-2+n+2r=3d\equiv 0$ mod 3. This shows that if $(g,n,r)$ is a valid tuple arising from a branched cover of a $g_0=2$ SRS with 1 puncture, then the congruence condition must be satisfied. To derive the genus bound, we want to minimize $g$ according to \eqref{hurwitz formula} with $g_0=2$ and given $n,r>0$. The minimal choice of the ramification pattern is
    \[
    \rho_{\text{min}}=(\underbrace{1,...,1}_{n},\underbrace{2,...,2}_{2r}).
    \]
    The corresponding minimal degree of the branched cover is $d_{\text{min}}=1\cdot n+2\cdot 2r=n+4r$. By the Hurwitz formula, the minimal genus is given by $2g_{\text{min}}-2=3d_{\text{min}}-n-2r$. Hence the solution is $g_{\text{min}}=n+5r+1$. Hence, the genus bound is also derived. This shows that for the tuple $(g,n,r)$ to be realizable, the genus bound and congruence condition are necessary.

    Now it remains to show that if these two conditions are met, then there exists a branched cover $\pi:Y\to X$ with the specified behavior.  Given the congruence condition, we note that the degree of the cover
    \[
    d=\frac{1}{3}(2g-2+n+2r)
    \]
    is an integer. Moreover, the genus bound $g\geq n+5r+1$ implies $2g-2\geq 2n+10r$. Substituting this into the expression above gives
    \[
    3d=2g-2+n+2r\geq (2n+10r)+n+2r=3n+12r.
    \]
    Hence $d\geq n+4r=d_{\text{min}}$. Therefore, $d$ is at least as large as the minimal possible degree $d_{\text{min}}$. 
    We must now show that there exists a partition $\rho$ of $d$ that has exactly $n$ odd parts and $2r$ even parts. The proof is constructive. Let $\delta_d=d-d_{\text{min}}$. The calculation above shows that
    \[
    3\delta_d=3(d-d_{\text{min}})=(2g-2+n+2r)-(2g_{\text{min}}-2+n+2r)=2(g-g_{\text{min}}).
    \]
    This implies that $3\delta_d$ is an even number. Since 3 is odd, $\delta_d$ itself must be an even number. Let $\delta_d=2k$ for some nonnegative integer $k$. We now need to find a partition of $d=d_{\text{min}}+2k$ with the correct number of even and odd parts. We start with the minimal partition $\rho_{\text{min}}$. We can modify this partition to increase its sum by an even number, $2k$, without changing the parity count of its parts. For example, we can replace a part $a_j=1$ with a part $a_j+2=3$. This increases the total sum by 2, and the new part is still odd, so the parity of the ramification pattern is preserved, while the total degree increases by 2. By repeatedly applying such modifications $k$ times, we can increase the sum of the partition from $d_{\text{min}}$ to $d=d_{\text{min}}+2k$ while preserving the number of even and odd parts. The resulting partition $\rho$ has sum $d$ and corresponds to the puncture configuration $(n,2r)$. By construction, this partition, when used in the Hurwitz formula, yields the genus $g$. To see this, we note that
    \begin{equation}\label{g_min hurwitz}
        g_{\text{min}}=1+d_{\text{min}}+\frac{1}{2}\sum_j (a_{j,\text{min}}-1).
    \end{equation}
    Since
    \[
    6k=3\delta_d=2(g-g_{\text{min}}),
    \]
    we have $g-g_{\text{min}}=3k$, and we also have
    \begin{equation}\label{hurwitz difference}
        (d-d_{\text{min}})+\frac{1}{2}\sum_{j=1}^s (a_j-a_{j,\text{min}})=3k=g-g_{\text{min}}.
    \end{equation}
    Now adding \eqref{g_min hurwitz} and \eqref{hurwitz difference} gives the Hurwitz formula
    \(
    g=1+d+\frac{1}{2}\sum_{j=1}^s (a_j-1),
    \)
    as desired. This concludes the proof.
\end{proof}
Now we aim to remove the congruence condition in Theorem \ref{main theorem} at the cost of a slightly stronger genus bound. Say we already know that $\mf M_{g,n,2r}$ is not projected, we consider the forgetful map
\[
F: \mf M_{g,n,2r}\to \mf M_{g,n-i,2r}
\]
for $i=1,2$, and show that the composition
\[
\Psi'=F\circ \Psi: \mf M_{d,\rho}\to \mf M_{g,n-i,2r}
\]
is still an immersion, and the bosonic normal bundle sequence splits. Since $\Phi: \mf M_{d,\rho}\to \mf M_{2,1,0}$ is still a finite covering map, $\mf M_{d,\rho}$ is not projected, we conclude that $\mf M_{g,n-i,2r}$ is not projected by Theorem \ref{cover of nonproject is nonproject} and Theorem \ref{immersion+NBS splits gives non-project}. Then, the congurence condition in Theorem \ref{main theorem} will be violated, so we can essentially remove this condition. 

In other words, the strategy is as follows: Given a triple $(g,n,r)$. Our goal is to show that $\mf M_{g,n,2r}$ is not projected. We want to find a helper tuple $(g,n_0,r)$ that satisfies the original two conditions in Theorem \ref{main theorem}, and such that we have the forgetful immersion $\Psi':\mf M_{d,\rho}\to\mf M_{g,n_0,2r}\to \mf M_{g,n,2r}$, which would then show that $\mf M_{g,n,2r}$ is not projected. Therefore, we must find a new genus bound for $(g,n,r)$ such that if the genus bound is met, then such a helper tuple $(g,n_0,r)$ is gauranteed to exist. For a given $(g,n,r)$ with fixed $g,r$, the helper $(g,n_0,2r)$ must satisfy
\[
n\leq n_0\leq g-5r-1
\]
where the first inequality is because we need the existence of the forgetful map, and the second inequality comes from the genus bound in Theorem \ref{main theorem}. Now the congruence condition says that we must have
\[
n_0\equiv -2g+2-2r\mod 3.
\]
Since $g$ and $r$ are fixed, the number $-2g+2-2r$ is also fixed. Thus, the problem reduces to the following: can we find an integer $n_0$ on the interval $[n,g-5r-1]$ with a specified residue mod 3? Clearly, this can be done if $[n,g-5r-1]$ contains at least three integers. In other words, we only need $g-5r-1-n+1\geq 3$, or equivalently $g\geq n+5r+3$. Hence, we obtain

\begin{thm}\label{main result}
Let $g,n,r$ be positive integers. The supermoduli space $\mf M_{g,n,2r}$ is not projected if $g\geq n+5r+3$.
\end{thm}

Now it remains to show that $\Psi'$ is an immersion and the bosonic normal bundle sequence splits for $i=1,2$, by Theorem \ref{immersion+NBS splits gives non-project}.

\begin{prop}\label{Psi' still immersion i=1,2}
    The morphism $\Psi':\mf M_{d,\rho}\xrightarrow{\Psi}\mf M_{g,n,2r}\xrightarrow{F}\mf M_{g,n-i,2r}$ is an immersion of supermanifolds, for $i=1,2$.
\end{prop}

\begin{proof}
    This follows immediately from Lemma 15 of \cite{DO23}, which states that $\Psi'$ composed with the forgetful map $\mf M_{g,n-i,2r}\to \mf{M}_{g,0,2r}$ is an immersion. Hence it follows that $\Psi'$ is an immersion.
\end{proof}

\begin{prop}
    The normal bundle sequence associated with the bosonic reduction of $$\Psi':\mf M_{d,\rho}\xrightarrow{\Psi}\mf M_{g,n,2r}\xrightarrow{F}\mf M_{g,n-i,2r}$$ splits, for $i=1,2$.
\end{prop}

\begin{proof}
    Since $\Psi:\mf M_{d,\rho}\to \mf M_{g,n,2r}$ is an immersion, its bosonic reduction $\psi: \mc{SM}_{d,\rho}\to \mc {SM}_{g,n,2r}$ is still an immersion. Moreover, we know the normal bundle sequence of $\psi$ splits by Proposition \ref{NBS splits}. The map $F:\mf M_{g,n,2r}\to \mf M_{g,n-i,2r}$ is a fibration, hence so is its bosonic reduction $f:\mc{SM}_{g,n,2r}\to \mc {SM}_{g,n-i,2r}$. Moreover, by Proposition \ref{Psi' still immersion i=1,2} we know $f\circ\psi: \mc{SM}_{d,\rho}\to \mc{SM}_{g,n-i,2r}$ is still an immersion. Applying the following lemma will conclude the proof.
\end{proof}

    The lemma was established in the proof of Theorem 1.3 of \cite[p.48]{DW13}. The proof of this statement is not very hard, so we give it here. 

    \begin{lem}
        Suppose that $i:X\to Y$ is an immersion, $f:Y\to Z$ is a fibration, such that $f\circ i:X\to Z$ is still an immersion. If the normal bundle sequence of $i$ splits, then the normal bundle sequence of $f\circ i$ will also split.
    \end{lem}

\begin{proof}
    The differential $\dd f: T_Y\to f^*T_Z$ gives a bundle map over $Y$, with kernel $T_{Y/Z}$. Pulling back along $i$ we get a bundle map $i^*\dd f: i^*T_Y\to (f\circ i)^*T_Z$ over $X$, with kernel $i^*T_{Y/Z}$. This gives a commutative diagram with exact rows given by the normal bundle sequences:
    \[\begin{tikzcd}
	0 & {T_X} & {i^*T_Y} & {N_{X,Y}} & 0 \\
	0 & {T_X} & {(f\circ i)^*T_Z} & {N_{X,Z}} & 0
	\arrow[from=1-1, to=1-2]
	\arrow[from=1-2, to=1-3]
	\arrow[no head, from=1-2, to=2-2]
	\arrow[shift right, no head, from=1-2, to=2-2]
	\arrow[from=1-3, to=1-4]
	\arrow["{i^*\dd f}"', from=1-3, to=2-3]
	\arrow[from=1-4, to=1-5]
	\arrow[from=1-4, to=2-4]
	\arrow[from=2-1, to=2-2]
	\arrow[from=2-2, to=2-3]
	\arrow[from=2-3, to=2-4]
	\arrow[from=2-4, to=2-5]
    \end{tikzcd}\]
    A direct application of the snake lemma shows that $\ker(N_{X,Y}\to N_{X,Z})\cong i^*T_{Y/Z}$. Now, we are given a splitting $s:N_{X,Y}\to i^*T_Y$ of the top row. Composing it with the quotient map we obtain a new map
    \[
    s':N_{X,Y}\to i^*T_Y\to i^*T_Y/i^*T_{Y/Z}.
    \]
    Because $f$ is a submersion, $i^*\dd f:i^*T_Y\to (f\circ i)^*T_Z$ is surjective, hence passing to the quotient, the map $N_{X,Y}\to N_{X,Z}$ is still surjective with kernel $i^*T_{Y/Z}$ as discussed above. Therefore, we see that $N_{X,Z}=N_{X,Y}/i^*T_{Y/Z}$. Therefore, the map $s'$ above factors through $i^*T_{Y/Z}$ and we get a map
    \[
    s':N_{X,Z}\to i^*T_Y/i^*T_{Y/Z}.
    \]
    But pulling back the relative tangent sequence $0\to T_{Y/Z}\to T_Y\to f^*T_Z\to 0$ gives a short exact sequence
    \[
    0\to i^*T_{Y/Z}\to i^*T_Y\to (f\circ i)^*T_Z\to 0.
    \]
    Hence we conclude that $(f\circ i)^*T_Z=i^*T_Y/i^*T_{Y/Z}$, and the map $s'$ is actually a map
    \[
    s':N_{X,Z}\to (f\circ i)^*T_Z,
    \]
    which we claim to be the desired splitting. Indeed, $s'$ is obtained by taking the right inverse $s$ for the map $N_{X,Y}\to i^*T_Y$ and then passing to the quotient, hence is a splitting. This concludes the proof.
\end{proof}

\section*{Acknowledgment}
The author thanks Ron Donagi for suggesting this problem and for his guidance. The author would also like to thank Tony Pantev, Nadia Ott, Edward Witten, David Kazhdan, Yuanyuan Shen, and Fanzhi Lu for reading an earlier version of this paper and providing valuable feedbacks. Last but not least, thanks to Daebeom Choi, Xingyu Meng, and Victor Alekseev for helpful discussions.


\begin{thebibliography}{9}

\bibitem{DW13} Ron Donagi, Edward Witten, \textit{Supermoduli Space is Not Projected}, 2013.

\bibitem{DW14} Ron Donagi, Edward Witten, \textit{Super Atiyah Classes and Obstructions to Splitting of Supermoduli Space}, 2014.

\bibitem{DO23} Ron Donagi, Nadia Ott, \textit{Supermoduli Space with Ramond Punctures is Not Projected}, 2023.

\bibitem{manin} Yuri Manin, \textit{Gauge Field Theory and Complex Geometry}, Springer, 1985.

\bibitem{W-SRSM} Edward Witten, \textit{Notes on Super Riemann Surfaces and Their Moduli}, 2013.\

\bibitem{W-moresuperstringpert} Edward Witten, \textit{More On Superstring Perturbation Theory: An Overview Of Superstring Perturbation Theory Via Super Riemann Surfaces}, 2013.

\bibitem{HP1}Eric D'Hoker, D.H. Phong, \textit{Two-Loop Superstrings I, Main Formulas}, 2001.

\bibitem{HP2} Eric D'Hoker, D.H. Phong, \textit{Lectures on Two-Loop Superstrings}, 2002.

\bibitem{gr} Paul Green, \textit{On Holomorphic Graded Manifolds}, PAMS v.85, 1982.

\bibitem{ber} Felix Berezin, \textit{Introduction to Superanalysis}, Springer, 1987.

\bibitem{huse} Dale Husemoller, \textit{Ramified Coverings of Riemann Surfaces}, Duke Math. J., Volume 29 (1962) no. 1.






\end{thebibliography}
\end{document}